\documentclass[11pt]{amsart}

\usepackage{enumerate}
\usepackage{amsthm,amsmath,amssymb,graphics}
\usepackage{graphicx}
\usepackage{psfrag}
\usepackage{color}

\newcommand{\field}[1]{\mathbb{#1}}
\newcommand{\R}{\field{R}}
\newcommand{\PP}{\field{P}}

\newcommand{\Hp}{\field{H}}

\newcommand{\C}{\field{C}}

\newcommand{\eps}{\varepsilon}
\newcommand{\ML}{\mathcal{ML}}
\newcommand{\PML}{\mathcal{PML}}
\newcommand{\acylindrical}{{\rm acylindrical}}
\newcommand{\Isom}{{\rm Isom}}

\newcommand{\Mod}{{\rm Mod}}
\newcommand{\window}{{\rm window}}
\newcommand{\wb}{{\rm wb}}

\newtheorem{theorem}{Theorem}[section]

\newtheorem{lemma}[theorem]{Lemma}

\theoremstyle{definition}
\newtheorem{remark}[theorem]{Remark}
\theoremstyle{plain}


\begin{document}
\title[]{The Double Limit Theorem and its legacy.}
\author{Cyril Lecuire}


\begin{abstract}
This chapter surveys recent and less recent results on convergence of Kleinian representations, following Thurston's Double Limit and ``$AH(\acylindrical)$ is compact" Theorems.    
\end{abstract}

\maketitle


\section{Introduction}

Although Kleinian groups were discovered  in the late 19th century (by Schottky, Klein and Poincar\'e), the story of the present chapter's topic really starts in the early sixties with the works of Ahlfors and Bers on quasi-conformal deformation of Fuchsian groups (\cite{ahlfors:finitely:generated}, \cite{ahlfors:bers}, \cite{bers:spaces:kleinian}). In particular, after further development by Maskit \cite{maskit:selfmaps} and Kra \cite{kra:kleinian}, it led to the parametrization of the space of quasi-conformal deformations by the conformal structure at infinity. Combined with later works of Marden \cite{marden:finitely} and Sullivan \cite{sullivan:stability} this provided coordinates for the interior of the deformation space $AH(\pi_1(M))$ usually called the Ahlfors--Bers coordinates. This also led to Bers' compactification of Teichm\"uller space, \cite{bers:boundaries}, who in particular introduced sequences of quasi-Fuchsian groups converging to non quasi-Fuchsian ones. Meanwhile J{\o}rgensen developped methods to study sequences of Kleinian groups, showing that discreteness is a closed property and isolating two types of convergence, which he called algebraic and geometric. In contrast to this rich theory of deformations of quasi-Fuchsian groups, Schottky groups and other infinite covolume convex cocompact Kleinian groups, Mostow showed in the late sixties that cocompact Kleinian groups are rigid, \cite{mostow}.
Then in the late seventies, Thurston revolutionized the world of low-dimensional geometry, introducing original and exotic tools to prove beautiful and unexpected new results. 

In an incomplete series of articles (\cite{thurston:hypI} and sequel) Thurston planned to present the arguments involved in the proof of the Geometrization Theorem for Haken Manifolds. Convergence of Kleinian representations plays a central role in each of the three existing papers: the main result of \cite{thurston:hypI} is that $AH(\acylindrical)$ is compact (Theorem \ref{acylindrical} in the present chapter), the Double Limit Theorem (Theorem \ref{limit double} below) is essential in \cite{thurston:hypII} and \cite{thurston:hypIII} is devoted to the Broken Windows Theorem (Theorem \ref{broken window}) and related results.\\


Now that the historical context has been set up, let us get more technical. The deformation space $AH(\pi_1(M))$ of a hyperbolic $3$-manifold $M$ is the set of discrete and faithful representations $\rho:\pi_1(M)\to PSL_2(\C)$ up to conjugacy, equipped with the quotient of the compact open topology. We will elaborate on the topology of $AH(\pi_1(M))$ in \textsection \ref{sec:definitions} and \textsection \ref{sec:applications}. For now let us consider two simple cases: when $M$ is a product $I\times S$ over a closed surface and $M$ is acylindrical. In both cases  the conformal structures at infinity provides us with a homeomorphism $q:\mathrm{int}(AH(\pi_1(M)))\to \mathcal{T}(\partial_{\chi<0} M)$ and we call $q(\rho)$ the {\em Ahlfors--Bers coordinates} of $\rho$. From this homeomorphism, we get that if a sequence $\{\rho_i\}\subset \mathrm{int}(AH(\pi_1(M)))$ has bounded Ahlfors--Bers coordinates, then $\{\rho_i\}$ has a converging subsequence. The question we address in this chapter is: What is the behaviour of a sequence with diverging Ahlfors--Bers coordinates?\\

When $M$ is acylindrical, a complete answer has been provided by Thurston (\cite{thurston:hypI}, and an alternate proof was given by Morgan--Shalen \cite{morgan:shalen:3}) with the following result:

\begin{theorem}[$AH(\acylindrical)$ Is Compact]	\label{acylindrical}
If $M$ is any compact acylindrical $3$--manifold with boundary, then $AH(\pi_1(M))$ is compact.
\end{theorem}

\newtheorem*{theorem:acylindrical}{Theorem \ref{acylindrical}}

Moreover, the fact that $AH(\pi_1(M))$ is compact characterizes acylindrical manifolds.\\

When $M=I\times S$, the Ahlfors--Bers coordinates are a pair of metrics $(\sigma^+,\sigma^-)\in\mathcal{T}(S)\times\mathcal{T}(S)$. A first condition for convergence comes from the works of Ahlfors and Bers (see Theorem \ref{ahlfors}):

\begin{theorem}
Let $\{\rho_i\}\subset AH(\pi_1(S))$ be a sequence of representations with Ahlfors--Bers coordinates $(\sigma_i^+,\sigma_i^-)$.  If $\{\sigma_i^+\}$ converges in the Teichm\"uller space $\mathcal{T}(S)$ then $\{\rho_i\}$ has a converging subsequence.
\end{theorem}

If we allow both coordinates $\sigma_i^+$ and $\sigma_i^-$ to diverge, then $\{\rho_i\}$ may not have a converging subsequence. To undertake a finer analysis, we need a way to quantify the behavior of diverging sequences. Thurston used his compactification of Teichm\"uller space by projective measured laminations in the celebrated Double Limit Theorem:

\begin{theorem}[Double Limit Theorem]		\label{limit double}
Let $S$ be a closed surface and let $\mu^+,\mu^-$ be two measured geodesic laminations that bind $S$. Then for any sequence $\{\sigma_i^+,\sigma_i^-\}$ in $\mathcal{T}(S)\times \mathcal{T}(S)$ converging to (the projective classes of) $(\mu^+,\mu^-)$ in $\overline{\mathcal{T}}(S)\times \overline{\mathcal{T}}(S)$, the sequence of quasi-Fuchsian representations with Ahlfors--Bers coordinates $(\sigma_i^+,\sigma_i^-)$ has a converging subsequence.
\end{theorem}

\newtheorem*{theorem:double}{Theorem \ref{limit double}}

Otal gave an alternative proof of this result in \cite{otal:fibre}. Thurston's, Morgan--Shalen's and Otal's proofs of Theorems \ref{acylindrical} and \ref{limit double} have seen adaptations and improvements by different authors which led to various generalizations. In this chapter, we will survey those generalizations and outline the arguments that are involved in their proofs.\\

We conclude this introduction with a plan of the chapter.
In the second Section we introduce deformation spaces and Thurston's and Culler--Morgan--Shalen's compactifications. In section \ref{sec:double limit}, we explain Thurston's and Otal's proofs of the Double Limit Theorem.  In the following section, we describe Thurston's and Morgan--Shalen's arguments leading to the proof of the compactness of $AH(\acylindrical)$ and its more general version, the Broken Window Only Theorem. Then we explain how to combine the Broken Window Only Theorem with the proof of the Double Limit Theorem to get a convergence Theorem for manifolds with incompressible boundary. 
In Section \ref{sec:compressible}, we describe in details progress 
that led to a general statement for all compact hyperbolic $3$-manifolds, answering a question of Thurston. In Section \ref{sec:necessity}, we mention the obstacles encountered when trying to relax the conditions in the Double Limit Theorem until they are necessary and sufficient. Then we describe a change of setting, using the curve complex to define 
 such necessary and sufficient conditions. Lastly in Section \ref{sec:applications}, we depict some of the applications of the theorems listed in this chapter, starting with Thurston's Hyperbolization Theorem.

\section{Compactifications of deformation spaces}    \label{sec:definitions}


\subsection{Definitions}

\subsubsection*{Deformation spaces}
Let $M$ be a compact $n$-manifold (we are only interested in the cases $n=2$ and $3$) and set $G=\pi_1(M)$. Let $\mathcal{D}(G)\subset Hom(G, \Isom^+(\Hp^d))$ denote the set of discrete and faithful representations. Given $\rho\in\mathcal{D}(G)$, the quotient $\Hp^d/\rho(G)$ is a complete hyperbolic $n$-manifold homotopy equivalent to $M$. We equip $Hom(G, \Isom^+(\Hp^d))$ (and hence $\mathcal{D}(G)$) with the compact open topology, so that $\rho_n\longrightarrow\rho$ if $\rho_n(g)\longrightarrow\rho(g)$ for any $g\in G$. This topology is also called the {\em algebraic topology}. Notice that when $G$ is not Abelian, $\mathcal{D}(G)$ is a closed subset (\cite{chuckrow}, \cite{jorgensen:transformations}). The group $\Isom^+(\Hp^d)$ acts properly discontinuously by conjugacy on $\mathcal{D}(G)$ and the quotient $AH(G)$ is the {\em deformation space} of $G$. $AH(G)$ is also the space of marked hyperbolic structures $(N,h)$ where $N$ is a complete hyperbolic $n$-manifold and $h:M\to N$ is a homotopy equivalence, modulo the equivalence relation $(N,h)\sim (N',h')$ if there is an isometry $\psi:N\to N'$ such that $h'$ is homotopic to $\psi\circ h$.

When $d=2$ and $\partial M=\emptyset$, $AH(G)=\mathcal{T}(M)\cup \mathcal{T}(\overline{M})$ is the union of two copies of the Teichm\"uller space of $M$.

When $d=3$, by Thurston's Hyperbolisation Theorem, $AH(G)\neq\emptyset$ if and only if $M$ is irreducible and atoroidal. Let us focus on this case making our way towards the Ahlfors--Bers coordinates mentioned in the introduction. To simplify the notation and statements, we will use the same notation for a conjugacy class in $AH(G)$ and a representative of this conjugacy class and we will assume that $M$ is orientable and that $\partial M$ contains no tori.

\subsubsection*{Ahlfors--Bers coordinates}
 Given $\rho\in AH(G)$, the group $\rho(G)$ acts by conformal transformations on $\hat\C=\partial_\infty\Hp^3$. Let $\Omega_\rho$ be the maximal invariant open subset on which this action is properly discontinuous. We say that $\rho$ is {\em convex cocompact} if $(\Hp^3\cup\Omega_\rho)/\rho(G)$ is compact (this is equivalent to more classical definitions, see \cite{marden:finitely}). By \cite{marden:finitely} and \cite{sullivan:stability}, $\rho$ is in the interior of $AH(G)$ if and only if it is convex cocompact. To each component $\mathcal{C}$ of $\mathrm{int}(AH(\pi_1(M)))$ is associated a pair $(N,h)$ where $N$ is a compact $3$-manifold and $h:M\to N$ is a homotopy equivalence (up to an equivalence relation, see \cite{acm}, here we only need a representative). Then, for each $\rho\in \mathcal{C}$ there is a homeomorphism $f_\rho:N\to (Hp^3\cup\Omega_\rho)/\rho(G)$ such that $(f_{\rho}\circ h)_*=\rho$. Since the only requirement on $f_\rho$ is  $(f_{\rho}\circ h)_*=\rho$, the isotopy class of $f_\rho$ is uniquely defined up to the action of the group $\Mod_0(N)$ of isotopy classes of orientation-preserving homeomorphisms of $N$ that are homotopic to the identity.
 
 Associating to a each representation $\rho\in\mathcal{C}$ its conformal structure at infinity $\Omega_\rho/\rho(G)$, we get a map $q:\mathcal{C}\to \mathcal{T}(\partial N)/\Mod_0(N)$. As mentioned in the introduction, by results of Ahlfors--Bers \cite{ahlfors:bers}, Bers \cite{bers:spaces:kleinian}, Maskit \cite{maskit:selfmaps} and Kra \cite{kra:kleinian}, $q$ is a homeomorphism. We call $q(\rho)$ the {\em Ahlfors--Bers coordinates} of $\rho$. When $M$ has incompressible boundary, $Mod_0(M)$ and $Mod_0(N)$ are trivial and $\mathcal{C}\approx \mathcal{T}(\partial N)$ is an open ball.
 
Notice that when $M=S\times I$, $M$ is acylindrical or $M$ is a handlebody, $\mathrm{int}(AH(\pi_1(M)))$ has only one component (corresponding to $(M,Id)$). The interested reader may refer to \cite{acm} or \cite{anderson:survey:deformation} for an enumeration of the  components of $AH(\pi_1(M))$ and $\mathrm{int}(AH(\pi_1(M)))$ in general.\\

To study sequences that do not converge in the interior of $AH(\pi_1(M))$, we want to describe how their Ahlfors--Bers coordinates diverge in Teichm\"uller space. This naturally leads us to introduce Thurston's compactification.

Before that, let us finish this section with a notation. In a compact connected $n$-manifold $M$, a closed curve $\gamma$ defines through its free homotopy class a conjugacy class in the fundamental group that we will also denote by $\gamma$. Given $\rho\in AH(\pi_1(M))$, we denote by $\ell_\rho(\gamma)$ the length in $\Hp^n/\rho(\pi_1(M))$ of the geodesic $\gamma^*_\rho$ in the free homotopy class corresponding to $\rho(\gamma)$.

\subsection{Thurston's compactification of Teichm\"uller space}

Thurston constructed a compactification of Teichm\"uller space by projective measured foliations or equivalently projective measured geodesic laminations. We will adopt the latter since it is better suited to applications and extensions to Kleinians group. Before proceeding, let us briefly mention that this compactification led to Thurston's celebrated classification of surface homeomorphisms (\cite[Theorem 2.5]{thurston:hypII}, see also \cite{flp} or \cite{flp:english}).

A {\em geodesic lamination} $L$ on a closed hyperbolic surface $S$ is heuristically a Hausdorff limit of multi-curves, i.e. disjoint unions of simple closed geodesics. The actual definition, which follows, encompasses a slightly larger set but in practice we will only consider such limits. A {\em geodesic lamination} is a compact set that is a (non-empty) disjoint union of complete embedded geodesics. Note that this definition can be made independent of the choice of metric on $S$, see \cite[Appendice]{otal:fibre} for example. 

A {\em measured geodesic lamination} $\lambda$ consists of a geodesic lamination $|\lambda|$ and a transverse measure on $|{\lambda}|$: any arc $k\cong [0,1]$ embedded in $S$ transverse to $|\lambda|$, such that  $\partial k\subset S-|\lambda|$, is endowed with a transverse measure $d\lambda$ such that:
 \begin{enumerate}[-]
 \item the support of $d\lambda |_{k}$ is $|\lambda|\cap k$;
 \item if an arc $k'$ can be homotoped to $k$ by a homotopy preserving $|\lambda|$ then $\int_k\! d\lambda=\int_{k'} d\lambda$.
 \end{enumerate}
The simplest case of measured geodesic laminations is a weighted simple closed geodesic $\delta c$, i.e. a simple closed geodesic $c$ equipped with a transverse Dirac measure with weight $\delta$. Weighted multi-curves are dense in the space $\mathcal{ML}(S)$ of measured geodesic laminations equipped with the weak$^*$ topology. Thus measured geodesic laminations can simply be viewed as limits of weighted multi-curves.

Given a hyperbolic metric on $S$, and hence a faithful and discrete representation $\rho:\pi_1(S)\to PSL_2(\R)$, the length of a weighted simple closed geodesic $\delta c$, is defined by homogeneity: $\ell_\rho(\delta c)=\delta\ell_\rho(c)$. Then the length of a weighted multi-curve is simply the sum of the length of its weighted leaves and the length of a measured geodesic lamination is defined by taking limits of lengths of weighted multi-curves. Alternatively, given a measured geodesic lamination $\mu$, we may pick a family $k$ of arcs transverse to its support $|\mu|$ so that the components of $|\mu|-{k}$ are arcs with bounded lengths. Then the length of $\mu$ is computed by integrating the lengths of these arcs over the transverse measure. It turns out that these two definitions are equivalent and it follows from this equivalence that the definition using limit of sequences of weighted multi-curves is independent of the choice of the sequence.

Given a simple closed geodesic $c$ and $\lambda\in\mathcal{ML}(S)$, the intersection number $i(c,\lambda)$ is the total weight of the measure on $c$ when $c$ is transverse to $\lambda$ and is $0$ otherwise, i.e. when $c$ lies in or is disjoint from the support of $\lambda$. This extends to weighted simple closed geodesics by homogeneity: $i(\delta c,\lambda)=\delta i(c,\lambda)$, to weighted multi-curves by additivity and then to measured geodesic laminations by continuity.

There is a natural action of $\R_+^*$ on $\mathcal{ML}(S)$ obtained by multiplying the measure and the space $\mathcal{PML}(S)$ of projective measured geodesic laminations is the quotient of $\mathcal{ML}(S)-\{0\}$ under this action.

Thurston uses the intersection number to define a compactification of Teichm\"uller space by projective measured geodesic laminations (\cite[Theorem 2.2]{thurston:bams}):

\begin{theorem}[Laminations compactify Teichm\"uller space]      \label{thurston comp}

The union $\overline{T(S)}=T(S)\cup \PML(S)$ has a natural topology homeomorphic to a closed ball.

In this topology, a sequence \textsc{$\{\rho_i\}$} of representations in $\mathcal{T}(S)$ converges to a lamination $[\mu]\in \PML(S)$ if and only if there is a sequence $\{\mu_i\}\longrightarrow\infty$ (i.e. there is an arc $k$ with $\int_kd\mu_i\longrightarrow\infty$) of measured laminations converging projectively to $\mu$ such that for all $\mu'\in \ML(S)$ for which $i(\mu',\mu)\neq 0$,
$$\lim_{i\longrightarrow\infty}\frac{\ell_{\rho_i}(\mu')}{i(\mu_i,\mu')}=1.$$
Furthermore, $\ell_{\rho_0}(\mu_i)\longrightarrow\infty$ but $\ell_{\rho_i}(\mu_i)$ remains bounded.

Moreover, there is a constant $C$ such that
\begin{equation}
i(\mu',\mu_i)\leq \ell_{\rho_i}(\mu')\leq i(\mu',\mu_i)+ C \ell_{\rho_0}(\mu').	 \label{thucom}
\end{equation}

\end{theorem}

The first part of the statement defines the compactification of Teichm\"uller space by projective measured geodesic lamination. The general idea is that if a sequence eventually stays outside every compact set, the lengths of some closed geodesics go to infinity: the metric is stretched. Since the area is bounded, locally, the metric is stretched only in one direction, transversely to a measured geodesic lamination $\mu_i$ so that $\lim_{i\longrightarrow\infty}\frac{\ell_{\rho_i}(\mu')}{i(\mu_i,\mu')}=1$. We may then extract a projectively converging subsequence from the sequence $\{\mu_i\}$.

Formula (\ref{thucom}) gives a more precise and uniform approximation. This uniformity can be used to prove a convergence result for surfaces in the spirit of the  Double Limit Theorem (compare with Theorem \ref{dl curves}). We say that two measured geodesic laminations $\gamma,\lambda$ {\em bind} $S$ if $i(\gamma,\nu)+i(\lambda,\nu)>0$ for any non-trivial $\nu\in\ML(S)$.

\begin{theorem}		\label{baby dl curves}
Let $S$ be a closed surface and let $\mu^\pm$ be two measured geodesic laminations that bind $S$. Let $\{\mu_i^\pm\}$ be two sequences of weighted multi-curves converging $\mu^\pm$. Then any sequence $\{\rho_i\}\subset \mathcal{T}(S)$ such that $\{\ell_{\rho_i}(\mu_i^+)\}$ and $\{\ell_{\rho_i}(\mu_i^-)\}$ are bounded has a converging subsequence.
\end{theorem}

\begin{proof}
If $\{\rho_i\}$ does not have a converging subsequence then it has a subsequence converging to a projective measured geodesic lamination $[\nu]$.  Theorem \ref{thurston comp} provides a sequence $\{\nu_i\}\longrightarrow\infty$ converging projectively to $\nu$, i.e. $\{\eps_i \nu_i\}$ converges to $\nu$ for a sequence $\eps_i\longrightarrow 0$, such that the inequalities (\ref{thucom}) are satisfied. Since $\mu^+$ and $\mu^-$ bind $S$, $i(\mu^+,\nu)+i(\mu^-,\nu)>0$, say $i(\mu^+,\nu)>0$. By continuity of the intersection number, $i(\mu_i^+,\nu_i)\longrightarrow\infty$. Now inequality (\ref{thucom}) contradicts the assumption that $\ell_{\rho_i}(\mu^+)$ is bounded.
\end{proof}

\subsection{Culler--Morgan--Shalen's compactification}    \label{compactification:tree}

A different point of view on the compactification of deformation spaces, using methods from algebraic geometry, was introduced by Culler and Shalen in \cite{culler:shalen} and then further developed by Morgan and Shalen (\cite{morgan:shalen:1}, \cite{morgan:shalen:2} and \cite{morgan:shalen:3}). In particular, in \cite{morgan:shalen:1} and \cite{morgan:compactification} (see also \cite{morgan:shalen:intro}), Morgan and Shalen use valuations to compactify deformation spaces for hyperbolic manifolds in any dimension. The added points are actions on $\Lambda$-trees from which one easily extracts an action on a real tree (more details about these below). By a result of Skora, \cite{skora:splitting}, (small minimal) actions of surface groups on real trees are dual to measured geodesic  laminations. Thus, in dimension $2$, Thurston's and Culler--Morgan--Shalen's compactification of Teichmüller spaces are equivalent.

In \cite{bestvina:degenerations}, \cite{paulin} and \cite{chiswell}, Bestvina, Paulin and Chiswell give an alternative and more geometric approach (with some variations) to Culler--Morgan--Shalen's compactification by actions on real trees. Let us sketch the ideas behind that geometric approach.

Consider a sequence of faithful and discrete representation $\rho_i:G\to Isom(\Hp^d)$ of a non-Abelian finitely generated group $G$ and set \( K_i=\inf\limits_{x\in\Hp^d}\{\max\limits_{g\in S} d(x,gx)\}\) for a finite generating set $S\subset G$. Since $\rho_i(G)$ is discrete and non-Abelian, $K_i$ is a minimum reached at some point $x_i$. Up to conjugating $\rho_i$, we may assume that $x_i=O$. The sequence $\{\rho_i\}$ stays in a compact subset of the deformation space if and only if $K_i$ is bounded. When $K_i$ goes to infinity, one rescales $\Hp^d$ by multiplying the distances by $K_i^{-1}$ so that the action of $\rho_i(S)$ on $K_i^{-1}\Hp^d$ is bounded. In $\Hp^d$, geodesic triangles are $\delta$-thin, in the sense that any edge lies in a $\delta$-neighbourhood of the other two (with $\delta=\log 2$). When we rescale the metric, the triangles become $K_i^{-1} \delta$-thin with $K_i^{-1} \delta\longrightarrow 0$, so that they look more and more like tripods as $i$ goes to $\infty$. One then just needs the appropriate formalism to find a subsequence such that the action of $\rho_i(G)$ on $K_i^{-1}\Hp^n$ tends in some way to  an action on a geodesic metric space where every geodesic triangle is a tripod. Such a space is called a {\em real tree}, a generalisation of simplicial trees that allows more flexibility on the vertices (they can accumulate or form a continuum). The convergence "in some way" is made formal by using the pointed Gromov--Hausdorff topology and either sequences of expanding finite subsets of $G$ (as in \cite{bestvina:degenerations} and \cite{paulin}) or ultra-filters (\cite{chiswell}, see also \cite[chapter 9]{kapovich:book}). Thus we have extracted a subsequence of $\rho_i$ converging to an action of $G$ on a real tree. Up to taking a substree, the action can be assumed to be {\em minimal}, i.e. there is no invariant subtree. Furthermore, one can deduce from Margulis' Lemma that the action is {\em small}, i.e. edge stabilizers are Abelian. Notice that if we choose a different generating set $S$, we may get a different sequence $K_i$ and the limiting tree may differ by a homothety.\\

As mentioned above, in dimension $2$, the compactification by actions on real trees is equivalent to Thurston's compactification by projective measured geodesic laminations. This identification goes through the dual tree $\mathcal{T}_\lambda$ to a measured geodesic laminations $\lambda\in\mathcal{ML}(S)$. To define $\mathcal{T}_\lambda$, we first replace the closed leaves by foliated neighbourhoods so that the transverse measure has no atoms. The preimage  $\tilde \lambda\subset \Hp^2$ of $\lambda$ under the covering projection $\Hp^2\to S$ defines a partition $\mathcal{P}$ of $\Hp^2$ into closed sets. An element of $\mathcal{P}$ is either the closure of a component of $\Hp^2-{|\lambda|}$ or a leaf of $|\lambda|$ which is not in the closure of such a component. The transverse measure defines a distance on $\mathcal{P}$ turning it into a real tree $\mathcal{T}_\lambda$ and the action of $\pi_1(S)$ on $\Hp^2=\tilde S$ induces an action on $\mathcal{T}_\lambda$. Notice that by the theorem of Skora (\cite{skora:splitting}), any small minimal action of $\pi_1(S)$ on a real tree is dual to a measured geodesic lamination.

If a sequence of representations $\rho_i:\pi_1(S)\to PSL(2,\R)$ converges in Thurston's compactification to a (projective) measured geodesic lamination $\mu$, then $\{\rho_i\}$ also converges in Culler--Morgan--Shalen's compactification to the action of $\pi_1(S)$ on $\mathcal{T}_\mu$. A simple way to see the unity of these two compactifications is to look at translation lengths and intersection numbers. Given an action of a group $G$ on a real tree $\mathcal{T}$ and $g\in G$, define its translation distance by $\delta_{\mathcal{T}}(g)=\inf\{d(x,gx)|x\in\mathcal{T}\}$. By \cite{culler:morgan}, a minimal action of $G$ by isometries on a real tree is uniquely defined by the function $\delta:G\to\R^+$. If $c$ is a simple closed curve on $S$ and if we also denote by $c$ the corresponding element of $\pi_1(S)$, then we have $\delta_{\mathcal{T}_\mu}(c)=i(\mu,c)$. Now in Thurston's compactification, we have $\eps_i\longrightarrow 0$ such that $\eps_i\ell_{\rho_i}(c)\longrightarrow i(c,\mu)$ while in Culler--Morgan--Shalen's, we have $\eps_i\ell_{\rho_i}(c)\longrightarrow \delta_{\mathcal{T}}(c)$. Hence $\mathcal{T}$ is dual to $\mu$.


\section{The Double Limit Theorem}  \label{sec:double limit}

In this section, we will describe Thurston's and Otal's proofs of the Double Limit Theorem. Let us first recall its statement.

\begin{theorem:double}
Let $S$ be a closed surface and let $\mu^+,\mu^-$ be two measured geodesic laminations that bind $S$. Then for any sequence $\{(\sigma_i^+,\sigma_i^-)\}$ in $\mathcal{T}(S)\times \mathcal{T}(S)$ converging to $(\mu^+,\mu^-)$ in $\overline{\mathcal{T}}(S)\times \overline{\mathcal{T}}(S)$, the sequence of quasi-Fuchsian representations with Ahlfors--Bers coordinates $(\sigma_i^+,\sigma_i^-)$ has a converging subsequence.
\end{theorem:double}

The first step in both proofs consists in establishing a link between the lengths of curves with respect to the conformal structures at infinity and their lengths inside the quotient $3$-manifold. Let $\rho:\pi_1(S)\to PSL_2(\C)$ be a quasi-Fuchsian representation with Ahlfors--Bers coordinates (i.e. conformal structures at infinity) $(\sigma^+,\sigma^-)\in \mathcal{T}(S)\times \mathcal{T}(S)$. Given a closed curve $\gamma\subset S$, let $\ell_{\sigma^+}(\gamma)$, resp. $\ell_{\sigma^-}(\gamma)$, denote the length of the geodesic in the homotopy class of $\gamma$ with respect to the metric $\sigma^+$, resp. $\sigma^-$. Let also $\ell_\rho(\gamma)$ denote the length of the geodesic of $\Hp^3/\rho(\pi_1(S))$ in the homotopy class defined by $\gamma$.

\begin{lemma}	\label{ahlfors}
We have: $\ell_\rho(\gamma)\leq 2\inf\{\ell_{\sigma^+}(\gamma),\ell_{\sigma^-}(\gamma)\}$.
\end{lemma}

This statement, which is a reformulation of \cite[Theorem 3]{bers:boundaries}, follows also from the work of Ahlfors \cite{ahlfors:finitely:generated} (see \cite[Lemma 5.1.1]{otal:fibre}).\\

If a sequence $\sigma^\pm_i$ converges to a lamination $\mu^\pm$, then by Theorem \ref{thurston comp}, there is a sequence of measured laminations $\mu_i^\pm$ converging projectively to $\mu^\pm$ such that $\ell_{\sigma_i^\pm}(\mu_i)$ remains bounded. Since weighted multi-curves are dense in $\ML (S)$, we can assume that $\mu_i^\pm$ is a multi-curve for any $i$. Combining this observation with Lemma \ref{ahlfors}, Theorem \ref{limit double} reduces to the following generalization of Theorem \ref{baby dl curves}:

\begin{theorem}		\label{dl curves}
Let $S$ be a closed surface and let $\mu^+,\mu^-$ be two measured geodesic laminations that bind $S$. Let $\{\mu^+_i\},\{\mu^-_i\}\subset\ML(S)\times\ML(S)$ be two sequences of weighted multi-curves converging respectively to $\mu^+$ and $\mu^-$. Then any sequence $\{\rho_i\}\subset AH(\pi_1(S))$ such that $\{\ell_{\rho_i}(\mu^+_i)\}$ and $\{\ell_{\rho_i}(\mu^-_i)\}$ are bounded has a converging subsequence.
\end{theorem}

\subsection{Thurston's arguments: efficiency of pleated surfaces}	\label{sec:thurston}

Thurston's approach to prove the Double limit Theorem is to project the $3$-manifold to an immersed surface especially constructed so that the induced distortion on the metric is controlled and thus reduce the problem to the $2$-dimensional case. This is done through the \lq\lq Efficiency of Pleated Surfaces" Theorem which allows one to estimate the lengths of geodesics in the $3$-manifold based on the length of their representatives on some surfaces specifically immersed in it. These surfaces are {\em pleated surfaces}, namely the immersions have totally geodesic image except on a geodesic lamination called the {\em pleating locus}. Such a pleated surface is locally ruled and the induced metric is hyperbolic. For example, let us pick a finite maximal lamination $\lambda\in\ML(S)$, and a representation $\rho\in \mathrm{int}(AH(\pi_1(S)))$. A surface $f_\lambda:S\to N_\rho=\Hp^3/\rho(\pi_1(S))$ pleated along $\lambda$ always exists, it maps the leaves of $\lambda$ to geodesics and the complementary regions to geodesic triangles.

\begin{theorem}[Efficiency of pleated surfaces, {\cite[Theorem 3.3]{thurston:hypII}}]	\label{efficiency}
Let $S$ be a closed surface. For any $\eps>0$, there is a constant $C<\infty$ such that the following holds:
\begin{enumerate}[]
\item Let $\lambda$ be any finite maximal lamination on $S$.
\item Let $\rho$ be any element of $\mathrm{int}(AH(\pi_1(S)))$ such that no closed leaf of $\lambda$ has length less than $\eps$ in $N_\rho=\Hp^3/\rho(\pi_1(S))$, and let $f_\lambda:S\to N_\rho$ be a surface which is pleated along $\lambda$.
\item Let $\mu\in ML(S)$ be a measured geodesic lamination.
\end{enumerate}
Then $$\ell_\rho(\mu)\leq \ell_{f_\lambda}(\mu)\leq \ell_\rho(\mu)+Ca(\lambda,\mu).$$

\end{theorem}


We will describe the {\em alternation number} $a ( \lambda ,.)$ in the sketch of the proof of Theorem \ref{efficiency}. 
For the proof of the Double Limit Theorem we only need to know that $a(\lambda ,.)$ is finite and continuous (\cite[Proposition 3.2]{thurston:hypI}).\\

Before describing the proof of Theorem \ref{efficiency}, let us explain how it is used to conclude the proof of Theorem \ref{dl curves} (and hence of the Double Limit Theorem).
First, Thurston produces a curve $c$  (actually infinitely many such curves, see \cite[Corollary 4.3]{thurston:hypII}) which is not too short in any of the manifolds $N_i=\Hp^3/\rho_i(\pi_1(M))$ (up to extracting a subsequence), i.e. $\ell_{\rho_i}(c)\geq\eps$ for all $i$ and a constant $\eps$ that depends only on $S$. Adding spiraling leaves, $c$ can easily be extended  to a maximal lamination $\lambda$ with no closed leaf except for $c$. Then there is a unique pleated surface $f_{\lambda,i}: S\to N_i$ which maps each component of $S-\lambda$ to a geodesic triangle. Applying Theorem \ref{efficiency}, we get that both $\{\ell_{f_{\lambda,i}}(\mu^+_n)\}$ and $\{\ell_{f_{\lambda,i}}(\mu^-_n)\}$ are bounded (see also \cite[Theorem 4.4]{thurston:hypII}). By Theorem \ref{baby dl curves}, the metric induced by $f_{\lambda,i}$ stays in a compact set. It follows that for any closed curve $d$ on $S$, $\ell_{\rho_i}(d)\leq \ell_{f_{\lambda,i}}(d)$ is bounded and that the sequence $\{\rho_i\}$ has a converging subsequence.

\begin{proof}[Sketch of the proof of Theorem \ref{efficiency}]
These inequalities need only be proved for simple closed curves. Then they holds for weighted multicurves and extend to measured lamination by continuity of the length function (\cite{brock:continuity}) and of the alternation number (\cite[Proposition 3.2]{thurston:hypI}). The left hand inequality is obvious so we focus on the right hand one.

Let $d\subset S$ be a closed geodesic for the metric induced by $f_\lambda$. Approximate $d$ on $S$ by a piecewise geodesic curve $p$ made up of segments in $\lambda$ and small jumps between those segments. To ensure that the jumps are small, we pick successive segments in asymptotic leaves of $\lambda$, and to have a control on the number of segments we pick non-successive segments in non-asymptotic leaves. The number of segments is then the number $a(\lambda,d)$ of times the direction of asymptoticity of the leaves of $\lambda$ changes as one goes around $d$. 

Next, consider a simplicial annulus $A$ joining $f_\lambda(p)$ to the geodesic $d^*\subset N_\rho$ in the homotopy class of $f_\lambda(d)$ and fix $\delta>0$. From each point $x\in f_\lambda(p)$ draw in $A$ an arc $A_x$ orthogonal to $f_\lambda(p)$ which either has length $\delta$ or hit $\partial A$ before reaching that length. By the Gauss--Bonnet formula, the contribution to the length of $f_\lambda(p)$ of the points $x$ for which $A_x$ has length $\delta$, $A_x$ hits $d^*$ or $x$ is close to a corner is at most $\ell_N(d)+O(a(\lambda,d))$. For the remaining points, $A_x$ is a shortcut in $N$, and the Uniform Injectivity Theorem  (Theorem \ref{uniform injectivity} below) says that there is a shortcut in $S$ joining the preimage of the endpoints of $A_x$. It is not difficult to ensure in the construction of $p$ that there are not too many such shortcuts. Thus we get $C$ depending only on $S$ and $\eps$ such that $\ell_{f_\lambda}(d)\leq \ell_{f_\lambda}(p)\leq \ell_N(d)+Ca(\lambda,d)$.

To have a complete overview of the proof, it remains to examine the Uniform Injectivity Theorem. Given a differentiable manifold $N$, let $\PP N$ denote the tangent line bundle.

\begin{theorem}[Uniform Injectivity, {\cite[Theorem 5.7]{thurston:hypI}}]   \label{uniform injectivity}
Let $\eps_0>0$ and let $S$ be a closed surface. Given a representation $\rho\in AH(\pi_1(S))$, a pleated map $f:S\to N_\rho=\Hp^3/\rho(\pi_1(S))$ which induces $\rho$ and a lamination $\lambda\subset S$ which is mapped geodesically by $f$, denote by $g:\lambda\to \PP M_\rho$ the canonical lifting. There is $\delta_0>0$ depending only on $\eps_0$ and $S$ such that  for any two points $x$ and $y\in\lambda$ whose injectivity radii are greater than $\eps_0$, if $d_f(x,y)\geq\eps_0$ then $d_{N_\rho}(f(x),f(y))\geq\delta_0$. 
\end{theorem}

The uniformity comes from a limit argument. Thurston first shows that $g$ is injective (\cite[Theorems 5.5 and 5.6]{thurston:hypI}) by contradiction. A non injective map $g$ would map two leaves of $\lambda$ to the same geodesic and hence their closures to the same set. This would produce two non-homotopic simple closed curves $c_1,c_2\subset S$ with the same image under $f$. This would contradict the assumption that $f$ induces $\rho$. From the injectivity he then goes to the uniform injectivity by establishing the compactness of pleated surfaces (in the appropriate topology).

\end{proof}


\subsection{Otal's proof: real trees and $\delta$-realization of train tracks}   \label{otal}

In his book on Thurston's Hyperbolization Theorem for manifolds which fiber over the circle, Otal introduces a different strategy to prove the Double Limit Theorem. It goes by contradiction, using the Culler--Morgan--Shalen compactification by actions on real trees (the geometric approach as described in \textsection \ref{compactification:tree}). The idea is to approximate geodesic laminations in $\Hp^3/\rho_i(\pi_1(S))$ by piecewise geodesic arcs with the geodesic pieces belonging to a finite set of homotopy classes which do not depend on $i$. The convergence to an action on a real tree allows us to estimate the behavior of the length of those geodesic arcs and then the behavior of geodesic laminations. These alternative arguments require an additional hypothesis: 

\begin{theorem}		\label{dl curves otal}
Let $S$ be a closed surface and let $\mu^+,\mu^-$ be two minimal measured geodesic laminations that bind $S$. Let $\{\mu^+_i,\mu^-_i\}$ be two sequences of weighted multi-curves converging in the Hausdorff topology to almost minimal laminations containing $\mu^+$ and $\mu^-$ respectively. Then any sequence $\{\rho_i\}\subset AH(\pi_1(S))$ such that $\{\ell_{\rho_i}(\mu^+_i)\}$ and $\{\ell_{\rho_i}(\mu^-_i)\}$ are bounded has a converging subsequence.
\end{theorem}

A geodesic lamination is {\em minimal} if any leaf is dense and {\em almost minimal} if it is made up of one minimal lamination $\mu$ and finitely many leaves accumulating on $\mu$. Notice that if $\mu^+$ and $\mu^-$ have simply connected complementary regions (for example when they are stable laminations of pseudo-Anosov mapping classes), we could equivalently request that $\{\mu^+_i\}$ and $\{\mu^-_i\}$ converge projectively to projective laminations supported by $|\mu^+|$ and $|\mu^-|$. In particular Theorem \ref{dl curves otal} is sufficient for the proof of the Hyperbolization Theorem.\\

As mentioned earlier, the proof goes by contradiction. We consider a sequence $\{\rho_i\}\subset AH(\pi_1(S))$ of quasi-Fuchsian representations tending to a small minimal action of $\pi_1(S)$ on a (projective)  real tree $\mathcal{T}$. Namely, there is $\varepsilon_i\longrightarrow 0$ such that the action of $\rho_i(\pi_1(S))$ on $\eps_i\Hp^3$ tends to the action of $\pi_1(S)$ on $\mathcal{T}$. By Skora's Theorem \cite{skora:splitting}, this action is dual to a (projective) measured lamination $\nu$. Since $\mu^+$ and $\mu^-$ bind $S$, at least one crosses $\nu$, say $i(\mu^+,\nu)>0$, and denote by $\mu^+_h$ the Hausdorff limit of $\{\mu^+_i\}$. The next step in the proof consists in constructing a train track carrying $\mu^+_h$ (and hence $\mu^+_i$ for $i$ large enough), using a segment of $\nu$ that crosses $\mu^+$ as its unique switch. Before that, let us take a short break to review some definitions.

 A {\em (fattened) train track} on a compact surface $S$ is a finite family of rectangles which intersect only at their vertical sides. A connected component of the union of the vertical sides is called a {\em switch} and such switches are required to be embedded arcs.  This is a fattened version of the train tracks defined by Thurston in \cite{thurston:notes} (see also Penner and Harer \cite{harer:penner}. The rectangles come with a vertical and a horizontal foliations. To carry the metaphor further, let us call {\em rail} a line made up of horizontal fibers and {\em tie} a leaf of the vertical foliation. A geodesic lamination is {\em carried} by a train track if (up to isotopy) it lies in the train track and is transverse to the ties.

Picking a segment $\kappa\subset|\nu|$  that crosses $\mu^+$ and grouping the component of $|\mu^+|-\kappa$ by homotopy classes, Otal constructs a train track $\mathcal{R}$ carrying $\mu^+_h$ with $\kappa$ as its only switch. The fact that $\mathcal{T}$ is dual to $\nu$ naturally produces a $\pi_1(S)$-equivariant map $f_\nu:\Hp^2\to \mathcal{T}$. By construction, this map $f_\nu$ is monotonous on the preimage of the rails of $\mathcal{R}$ and not constant on any rectangle. Otal uses this observation to turn $f_\nu$ into a {\em realization} of $\mathcal{R}$, i.e. a map $f$ that is injective when restricted to a lift of a rail. Then $f$ is also a realization of any geodesic lamination $\lambda$ carried by $\mathcal{R}$, i.e. it is injective when restricted to a leaf of the preimage of $\lambda$.

Let $\tilde{\mathcal{R}}\subset\Hp^2$ be the preimage of $\mathcal{R}$ and let $\tilde\kappa\subset\tilde{\mathcal{R}}$ be a lift of $\kappa$ (the switch of $\mathcal{R}$). Recall that the action $\rho_i(\pi_1(S))$ on $\eps_i\Hp^3$ tends to the action of $\pi_1(S)$ on $\mathcal{T}$ and consider a sequence of points $p_i\in\Hp^3$ tending to $p=f(\tilde\kappa)$. Consider the $\rho_i$-equivariant map $F_i:\tilde{\mathcal{R}}\to\Hp^3$ that maps $\tilde\kappa$ to $p_i$ and each rectangle to a geodesic segment. For any rectangle $\tilde{R}$, $\eps_i\ell (F_i(\tilde{R}))$ converges to the positive length of $f(\tilde{R})$. It follows that for any geodesic $l$ carried by $R$, $F(\tilde{l})$ is made up of long geodesic segments. But we cannot guarantee that $F(\tilde{l})$ is a quasi-geodesic since we have no control on the angle between two successive geodesic segments. 

In the last step of the proof, Otal changes the train track $\mathcal{R}$ by a subdivision operation, producing a new train track $\mathcal{R}'$ carrying $\mu^+_h$ and a $\rho_i$-equivariant map $F'_i:\tilde{\mathcal{R}'}\to \Hp^3$ which maps rectangles to long segments such that the angles between two successive segments are close to $\pi$. Then for $i$ large enough and for any closed curve $c$ carried by $R$, the projection of $F'_i(\tilde{c})$ to $\Hp^3/\rho_i(\pi_1(S))$ is a quasi-geodesic and its length is close to the length of the geodesic $c_i^*\subset \Hp^3/\rho_i(\pi_1(S))$ in the same homotopy class. Thus the length of $c_i^*$ is approximated by the sum of the lengths of the images of the rectangle of $\mathcal{R}$ it goes through and we get: 
$$\varepsilon_i \ell_{\rho_i}(c_i^*)\geq K  \ell_{s_0}(c).$$
where $\ell_{s_0}$ is the length for a fixed reference hyperbolic metric on $S$, a simple way to roughly evaluate the number of rectangles through which $c$ goes, $K$ is a constant that depends only on $\mathcal{R}$ and the inequality holds for $i$ large enough and for any closed curve carried by $\mathcal{R}$.

In particular, we have $\ell_{\rho_i}(\mu^+_i)\longrightarrow\infty$ which is the desired contradiction.

\begin{remark}  \label{almost max}
The assumption that $\mu^\pm_h$ is almost minimal was used in two instances:
\begin{enumerate}[-]
\item to deduce that $\mu^+_h$ is carried by a train track $\mathcal{R}$ realized in $\mathcal{T}$ from the assumption that $\lambda$ intersects $\mu$ and
\item to construct a train track $\mathcal{R}$ with only one switch carrying $\mu^+_h$.
\end{enumerate}
The fact that $\mathcal{R}$ has only one switch simplifies the construction but removing that constraint would only add more technicalities, whereas $\mathcal{R}$ being realized (or equivalently $\mu^+_h$ being realized) is required to end up with a piecewise geodesic curve made up of long segments with incident angles close to $\pi$.

Thus we could relax the assumption on $\mu^\pm$ being almost minimal as long as we can guarantee that $\mu^+_h$ or $\mu^-_h$ is realized in any dual tree.
\end{remark}

We could also put aside Skora's Theorem and dual laminations and start from the assumption that $\mu^+$ is realized in $\mathcal{T}$. Proceeding with the same arguments from that point on leads to:

\begin{theorem}[Continuity Theorem]	\label{contin}
Let $M$ be a compact atoroidal $3$-manifold and $\{\rho_i\}\subset AH(\pi_1(M))$ be a sequence tending to a small minimal action of $\pi_1(M)$ on a real-tree $\mathcal{T}$. Let $\varepsilon_i\longrightarrow 0$ be such that $\forall g\in\pi_1(M)$, $\varepsilon_i \delta_{\rho_i}(g)\longrightarrow \delta_\mathcal{T}(g)$ and let $\mu\subset \partial M$ be a geodesic lamination which is realized in $\mathcal{T}$. Then there exists a neighborhood $\mathcal{V}(\mu)$ of $|\mu|$, and constants $K,i_0$ such that for any simple closed curve $c\subset \mathcal{V}(\mu)$ and for any $i\geq i_0$,
$$\varepsilon_i l_{\rho_i}(c^*)\geq K  l_{s_0}(c).$$
\end{theorem}

\section{Manifolds with incompressible boundary}    \label{secn:incompressible}


Next, we will consider Kleinian representations of fundamental groups of $3$-manifolds with incompressible boundary, starting with acylindrical manifolds.
Let us recall that an {\em essential} disc, annulus or torus is an incompressible properly embedded disc, annulus or torus that is not boundary parallel, i.e. cannot be homotoped relative to its boundary in $\partial M$. A compact $3$-manifold is {\em atoroidal} if it does not contain any essential torus and is {\em acylindrical} if it does not contain any essential disc, torus or annulus. \\

Before discussing the compactness of $AH(acylindrical)$, let us outline the importance of acylindrical manifolds in the topology of $3$-manifolds. For this purpose, we introduce the theory of the  characteristic submanifold (or JSJ decomposition). To give a general idea let us say that the {\em characteristic submanifold} $\Sigma$ of a compact $3$-manifold with incompressible boundary is the smallest submanifold that contains all the essential tori, Klein bottles, annuli and M\"obius bands up to isotopy (a precise definition can be found in \cite{johannson:jsj} and \cite{jaco:shalen:seifert}, see also \cite[Theorem 3.8]{bonahon:3-var}). Its existence and uniqueness (up to isotopy) has been established independently by Johannson \cite{johannson:jsj} and Jaco--Shalen \cite{jaco:shalen:seifert}. We are only interested in orientable atoroidal $3$-manifolds, in which case the components of $\Sigma$ are essential $I$-bundles, solid tori and thickened tori. The solid tori and thickened tori are only required to intersect $\partial M$ along a collection of disjoint annuli and tori, which is why they are not viewed as essential $I$-bundles. The components of $M-\Sigma$ are acylindrical relative to $\partial M$, i.e. if $W$ is the closure of a component of $M-\Sigma$ and $\partial_0 W=W\cap \partial M$ then any annulus $(A,\partial A)\subset (W,\partial_0 W)$ can be homotoped in $\partial W$ relative to its boundary.
A relative version of this theory produces a characteristic submanifold relative to an incompressible subsurface $\partial_0 M\subset\partial M$ of the boundary (see \cite[\textsection IV.4.]{morgan:shalen:2}), it contains all the essential annuli $(A,\partial A)\subset (M,\partial_0 M)$. This will be especially interesting in the next section where we will consider more general $3$-manifolds since $\partial M$ is allowed to be compressible as long as $\partial_0 M$ is incompressible.

Let us draw a simple conclusion from this dense paragraph: a compact orientable atoroidal $3$-manifold with incompressible boundary is made up of $I$-bundles, (relative) acylindrical submanifolds and submanifolds with abelian fundamental groups. Since we have already studied deformations of hyperbolic $I$-bundles in the previous sections, it now seems natural to follow up with acylindrical manifolds. 

\begin{theorem:acylindrical}[$AH(\acylindrical)$ Is Compact]
If $M$ is any compact acylindrical $3$-manifold with boundary, then $AH(\pi_1(M))$ is compact.
\end{theorem:acylindrical}

This result is due to Thurston, \cite[Theorem 1.2]{thurston:hypI}, and then was proved by Morgan--Shalen, \cite[Theorem V.2.1]{morgan:shalen:3} with very distinct ideas and tools. Their overall strategies also differ: Thurston first proves Theorem \ref{acylindrical} in \cite{thurston:hypI} and later introduces new arguments (in \cite{thurston:hypIII}) to extend of the proof to a more general setting whereas Morgan and Shalen directly prove a general statement in \cite{morgan:shalen:3} and deduce Theorem \ref{acylindrical} as a special case. Both strategies still lead to comparable generalizations, which essentially state the following: for a compact atoroidal $3$-manifold $M$ with incompressible boundary, a sequence in $AH(\pi_1(M))$ can only degenerate on the fundamental group of the characteristic submanifold.\\

\subsection{Thurston's proof and generalizations: degenerating simplices and broken windows}
Following the chronological order, let us first outline Thurston's proof of Theorem \ref{acylindrical}. Consider a sequence of maps $f_i:M\to M_i=\Hp^3/\rho_i(\pi_1(M))$ mapping a fixed triangulation of $M$ minus the vertices to ideal simplices so that the restriction to the boundary is a pleated surface. We separate the simplices of the triangulation of $M$ into two families $\Delta_b$ and $\Delta_\infty$ depending on whether the geometry of $f_i$ remains bounded or goes to infinity. Thurston deduces then from the Uniform Injectivity Theorem that a neighbourhood of the interface between these two families has boundary with small area and hence with Abelian fundamental group. It follows then from topological considerations that $\Delta_b$ carries the fundamental group. Thus the sequence $\{\rho_i\}$ is bounded.
In a subsequent paper, \cite{thurston:hypIII}, Thurston uses the same argument to establish a relative compactness Theorem:

\begin{theorem}[Relative Boundedness, {\cite[Theorem 3.1]{thurston:hypIII}}]	\label{rel bound}
Let $M$ be a $3$-manifold, and $\gamma$ a doubly incompressible multicurve on $\partial M$. Then for any constant $A>0$, the subset of $AH(\pi_1(M))$ such that the total length of $\gamma$ does not exceed $A$ is compact.
\end{theorem}

We say that a multi-curve on the boundary of a compact $3$-manifold is {\em doubly incompressible} if it intersects the boundary of any essential disc or annulus (this is a special case of Thurston's original definition \cite[p. 10]{thurston:hypIII} where $S=\partial M$ and $f$ is the inclusion).

Since $\partial M$ is not assumed to be incompressible, the Uniform Injectivity Theorem may not apply under the assumptions of Theorem \ref{rel bound} (we will give more insight on this issue in the next section). To overcome this difficulty Thurston extends the Uniform Injectivity Theorem at the price of loosing some of its uniformity: the constant $\delta$ depends also on a doubly incompressible multicurve $\gamma$ that must be contained in the pleating locus and on a bound on the length of this multicurve. Once this is established, the proof of Theorem \ref{rel bound} follows the same outline as the proof of Theorem \ref{acylindrical}.\\

Thurston uses Theorem \ref{rel bound} for a final generalization of Theorem \ref{acylindrical}: the Broken Windows Only Theorem. He uses a slight variation on the characteristic submanifold made up only of $I$-bundles which he calls the {\em windows}: he does not take the solid tori and thickened tori and replace them with a collection of thickened annuli. In his usual picturesque style, Thurston derives the name from the idea that if the manifold was made of glass, the window would be the part through which one could see without distortion. He shows that for a sequence in $AH(\pi_1(M))$, degenerations can only happen on the fundamental group of the window, hence carrying the metaphor further: \lq\lq only the window breaks".

\begin{theorem}[Broken Windows Only, {\cite[Theorem 0.1]{thurston:hypIII}}]		\label{broken window}
If $\Gamma\subset\pi_1(M)$ is any subgroup which is conjugate to the fundamental group of a component of $M-\window(M)$, then the set of representations of $\Gamma$  in $\mathrm{Isom}(\Hp^3)$ induced from $AH(\pi_1(M))$ are bounded, up to conjugacy.


\end{theorem}

The window is an $I$-bundle over a (usually disconnected) compact surface $S$ called the {\em window base} (denoted wb above). Its boundary $\partial S$ is the {\em window frame}.

The Broken Windows Only Theorem is deduced from Theorem \ref{rel bound} and a uniform bound on the length of the window frame:

\begin{theorem}[Window Frame Bounded, {\cite[Theorem 1.3]{thurston:hypIII}}] \label{frame}
For any manifold $M$ with incompressible boundary, there is a constant $C$ such that among all elements $\rho\in AH(\pi_1(M))$, the length in $N_\rho$ of $\partial \wb(M)$ is bounded.
\end{theorem}

Thurston's proof of Theorem \ref{frame} (\cite{thurston:hypI}, see also the appendix of \cite{morgan:compactification}) uses the area growth rate of branched pleated surfaces. An alternate proof using the Uniform Injectivity Theorem appeared in \cite[Appendix]{bmns:bounded:combinatorics}.\\

In \cite{thurston:hypII}, the Broken Windows Only Theorem (Theorem \ref{broken window}) has a second part, generalizing
 a previous result of Thurston on surface groups (\cite[Theorem 6.2]{thurston:hypII}), and setting up the existence of sequences of maximal convergence and submanifolds of maximal convergence. But, as observed by Ohshika, the Convergence on Subsurfaces Theorem, \cite[Theorem 6.2]{thurston:hypII}, does not extend to manifolds with incompressible boundaries as described by Thurston (see the example in \cite[\textsection 5.3]{ohshika:degeneration}). On the other hand, Canary, Minsky and Taylor ( \cite[Theorem 5.5]{canary:minsky:taylor}) observed that one may remark the representations and extract a subsequence so that it converges on most of $M$:

\begin{theorem}[{\cite[Theorem 2.8]{bbcm:windows}}]
Let $M$ be a compact $3$-manifold with incompressible boundary and consider a sequence $\{\rho_i\}\subset AH(\pi_1(M))$ of representation uniformizing $M$. Then after passing to a subsequence, there is a collection $\mathcal{B}$ of essential annuli and a sequence of homeomorphisms $\phi_i:M\to M$ each supported on $window(M)$ such that
\begin{enumerate}[(1)]
\item $\lim \ell_{\rho_n\circ(\phi_n)_*}(c)=0$ for any simple closed $c\subset\partial\mathcal{B}$ and
\item $\{\rho_n\circ(\phi_n)_*\}$ converges on the fundamental group of each component of $M- \mathcal{B}$.
\end{enumerate}

\end{theorem}

The proof of the last statement combines the Broken Windows Only Theorem, the Efficiency of Pleated Surfaces and Mumford Compactness Theorem (\cite{mumford:compactness}, see also \cite[Proposition 5.6]{canary:minsky:taylor}).

\subsection{Morgan and Shalen's arguments: trees and codimension-$1$ laminations}
Morgan and Shalen start in a very general setting by considering a compact irreducible $3$-manifold $M$ and an incompressible subsurface of its boundary $\partial_0 M\subset\partial M$. They associate to each small minimal action of $\pi_1(M)$ on a real tree $\mathcal{T}$ 
a measured codimension $1$ lamination $\mathcal{L}\subset M$ and a morphism between its dual tree $\mathcal{T}_\mathcal{L}$ and $\mathcal{T}$. This morphism may not be injective: it may fold, i.e. map two adjacent segments onto one. This possible lack of injectivity cannot be overcome in general since there are small minimal actions of fundamental groups of compact atoroidal $3$-manifolds on real trees which are not dual to any codimension $1$ laminations (see \cite{ohshika:degeneration}). A morphism is still enough to guarantee that the fundamental group of every component of $M- \mathcal{L}$ has a fixed point in $\mathcal{T}$. In a previous work \cite{morgan:shalen:2}, Morgan and Shalen have shown that such a lamination sits (up to some surgeries and isotopies) in the characteristic submanifold relative to $\partial M-\partial_0 M$. This leads to the following statement:

\begin{theorem}[{\cite[Theorem IV.1.2]{morgan:shalen:3}}]	\label{ms broken window}
Let $M$ be a compact irreducible $3$-manifold, let $\partial_0 M\subset M$ be an incompressible subsurface and let $\Sigma\subset M$ be the characteristic submanifold relative to $\partial_0 M$. Let $\pi_1(M)\times \mathcal{T}\to \mathcal{T}$ be a small action on a real tree and suppose that for any component $Z$ of $\partial M-\partial_0 M$, $\pi_1(Z)$ has a fixed point. Then for each component $C$ of $M-\Sigma$, the group $\pi_1(C)$ has a fixed point in $T$.
\end{theorem}

When $M$ is acylindrical, the characteristic submanifold is empty and it follows from this statement that there is no small minimal action of $\pi_1(M)$ on a non-trivial real tree. Then the conclusion of Theorem \ref{acylindrical} follows from Culler--Morgan--Shalen's compactification of the deformation space.

Let us add that a more general result about splitting of groups acting on real trees (from which Theorem \ref{ms broken window} can be deduced) has been obtained by Rips, using combinatorial methods instead of topological arguments, see \cite{bestvina:feighn:rips} and \cite[\textsection 12]{kapovich:book}.\\

\subsection{Mixing the arguments}
Theorems \ref{ms broken window} and \ref{broken window} both tell us that to bound a sequence of representations $\rho_i\in AH(\pi_1(M))$ it suffices to bound its restriction to the fundamental group of the window. Using this observation, we will extend the Double Limit Theorem to manifolds with incompressible boundaries. Let us first set up a property of laminations on $\partial M$ that will play the role of the binding property in the Double Limit Theorem. We say that a measured lamination $\lambda\in \ML(\partial M)$ on the boundary of a manifold with incompressible boundary is {\em acylindrical} if there is $\eps>0$ such that $i(\lambda,\partial A)\geq\eps$ for any essential annulus $A\subset M$. As observed by Bonahon--Otal (\cite{bonahon:otal:plissage}, when $M$ is not an $I$-bundle, it is equivalent to require that $i(\lambda,\partial A)>0$ for any essential annulus $A\subset M$.

\begin{theorem}
Let $M$ be a compact hyperbolizable $3$-manifold with incompressible boundary, let $\mathcal{C}$ be a connected component of $\mathrm{int}(AH(\pi_1(M)))$ containing a representation uniformizing $M$ and let $\mu\in\ML (\partial M)$ be an acylindrical measured geodesic lamination. Then for any sequence $\{\sigma_i\}$ in $\mathcal{T}(\partial M)$ converging to $\mu$ in $\overline{\mathcal{T}}(\partial M)$, the sequence of convex cocompact representations in $\mathcal{C}$ with Ahlfors--Bers coordinates $\sigma_i$ has a converging subsequence.
\end{theorem}


Using Lemma \ref{ahlfors}, the hypothesis on $\{\sigma_i\}$ can be replaced with a bound on the length of a sequence of weighted multi-curves converging to $\mu$. The resulting statement can then be established using Theorem \ref{broken window} and the arguments explained in \textsection \ref{sec:thurston}. If we add the assumption that the limit is almost minimal (see Remark \ref{almost max}), we can also build a proof on Theorem \ref{ms broken window} and Otal's arguments (compare with \cite[Lemme 14]{bonahon:otal:plissage}). Let us mix the two approaches to provide an alternative and fairly short proof (see also \cite[Theorem 3.7]{ohshika:limitqc} and \cite[Theorem 3.1]{ohshika:constructing} for different mixes of those arguments).

\begin{proof}
As in the proof of the Double Limit Theorem, we use Theorem \ref{thurston comp} and Lemma \ref{ahlfors} to obtain a sequence of weighted multi-curves $\mu_i\in \ML(\partial M)$ such that $\mu_i\longrightarrow\mu$ and $\{\ell_{\rho_i}(\mu_i)\}$ is bounded (compare with the beginning of \textsection \ref{sec:double limit}).

As we have seen in \textsection \ref{compactification:tree}, if $\rho_i$ has no converging subsequence then a subsequence converges to a small minimal action on a real tree $\mathcal{T}$, namely there is $\eps_i\longrightarrow 0$ such that $\eps_i\ell_{\rho_i}(c^*)\longrightarrow \delta_\mathcal{T}(c)$ for any closed curve $c\in M$. For each component $S$ of $\partial M$ with negative Euler characteristic, since $M$ has incompressible boundary, the map $i_*:\pi_1(S)\to \pi_1(M)$ induced by the inclusion provides us with a small action of $\pi_1(S)$ on $\mathcal{T}$. We can apply Skora's Theorem \cite{skora:splitting} to the minimal invariant subtree $\mathcal{T}_S\subset \mathcal{T}$ to get a dual lamination $\nu_S$.

Building a pleated surface $f_{\lambda,i}:S\to N_i=\Hp^3/\rho_i(\pi_1(M))$ with a pleating locus that never gets too short (as explained in the proof of Theorem \ref{dl curves}), we get from the Efficiency of Pleated Surfaces (Theorem \ref{efficiency}) $\ell_{\rho_i}(d)\leq \ell_{f_{\lambda,i}}(d)\leq \ell_{\rho_i}(d)+Ca(\lambda_i,d)$ for any simple closed curve $d\subset S$. In particular $\eps_i \ell_{f_{\lambda,i}}(d)\longrightarrow \delta_T(d)=\delta_{T_S}(d)$. It follows that the metric induced by $f_{\lambda,i}$ converges to $\nu_S$ in Thurston's compactification. In particular, there is a sequence $\{\nu_i\}\longrightarrow\infty$ of measured laminations converging projectively to $\nu$ such that $i(\gamma,\nu_i)\leq \ell_{f_{\lambda,i}}(\gamma)\leq i(\gamma,\nu_i)+ C' \ell_{f_0}(\gamma)$ for any measured lamination $\gamma$ on $S$. Combining these inequalities with the Efficiency of pleated surfaces, we get 
\begin{equation}
i(\gamma,\nu_i)-Ca(\lambda_i,\gamma)\leq \ell_{\rho_i}(\gamma)\leq i(\gamma,\nu_i)+ C'\ell_{f_0}(\gamma). \label{ineq:incomp}
\end{equation}

Set $\nu=\bigcup_{S\subset\partial M} \nu_S$ and denote by $S(\nu)$ its minimal supporting surface. Let $F\subset\partial M$ be an essential subsurface. It follows from the definition of $\nu$ that $\pi_1(F)$ has a fixed point in $T$ if and only if $F$ is disjoint from $S(\nu)$ (up to isotopy). It follows then from \cite[Theorem IV.1.2]{morgan:shalen:3} that there is a collection $\Sigma_\nu$ of essential $I$-bundles, solid tori and thickened tori such that $S(\nu)=\partial\Sigma\cap \partial M$. Assuming that $M$ is not an $I$-bundle, consider an essential annulus $A\subset \partial\Sigma$. By assumption, $i(\partial A,\mu)>0$. This is possible only if $i(\mu,\nu)>0$. Then we get $i(\mu_i,\nu_i)\longrightarrow\infty$ and $\ell_{\rho_i}(\mu_i)\longrightarrow\infty$ by inequality (\ref{ineq:incomp}). This contradiction concludes the proof.
\end{proof}

\section{Manifolds with compressible boundary}  \label{sec:compressible}

In the previous section we saw that with some additional work, an analogue of the Double Limit Theorem could be established for manifolds with incompressible boundary. To prove a similar result in full generality, we need to consider manifolds with compressible boundary. As we will see in this section, some of the results that were crucial in each proof either are not known in this level of generality or fail to be true.

The first step in both Thurston's and Otal's proof was Lemma \ref{ahlfors} and an essential hypothesis in its proof (see \cite[Theorem 3]{bers:boundaries} and \cite[Theorem 5.1.1]{otal:fibre}) is that the domain of discontinuity is simply connected. Canary first overcame this issue in  \cite{canary:poincare} by using new arguments and allowing the multiplicative constant to depend on the injectivity radius of the domain of discontinuity. 

\begin{theorem}[{\cite{canary:poincare}}]		\label{canahl}
Given $A>0$, there exists $R$ such that, if $\Gamma$ is a nonelementary Kleinian group such that every geodesic in $D_\Gamma$ has length (in the Poincar\'e metric on $D_\Gamma$) at least $A$ and if $c$ is any closed curve on $S=D_\Gamma/\Gamma$, then
$$\ell_N(c^*)\leq R\ell_S(c)$$
where $N=\Hp^3/\Gamma$.
\end{theorem}

Notice that a geodesic in the domain of discontinuity is a meridian, i.e. it bounds an essential disk. One can prove that a sequence for which the length of a meridian goes to $0$ necessarily diverges. Thus the dependence of the constant on $A$ will not be an obstacle when proving convergence results. Furthermore, at the price of dropping the linearity, Sugawa obtained in \cite[Proposition 6.1]{sugawa:perfectness} a universal constant with the following inequality (with the notation of Theorem \ref{canahl}):

\begin{equation}
\ell_N(c^*)\leq 2\ell_S(c) e^{\ell_S(c)/2}		\label{sugawa}
\end{equation}

This definitively solves the issue of replacing Lemma \ref{ahlfors}, even though, as mentioned earlier, Theorem \ref{canahl} was already enough. \\

A more critical obstacle when attempting to extend Thurston's arguments is that the Uniform Injectivity Theorem, which is essential in Thurston's proof of both the Double Limit Theorem and the compactness of $AH(\acylindrical)$, does not hold for compressible pleated surfaces. 
A sequence of compressible pleated surfaces for which the length of a meridian goes to $0$ does not converge in any reasonable sense. A way to get around this obstacle was given previously with the Relative Boundedness Theorem (Theorem \ref{rel bound}) where we required a bound on the length of a fixed doubly incompressible multicurve. Combining Theorem \ref{rel bound} and Sugawa's inequality (\ref{sugawa}), we get:


\begin{theorem}		\label{bounded infinity}
Let $\gamma\in\partial M$ be a doubly incompressible multicurve and consider a sequence $\{\rho_i\}\in AH(\pi_1(M))$. If $\ell_{\sigma_i}(\gamma)$ is bounded, then $\rho_i$ has a converging subsequence.
\end{theorem}

The idea of adding a bound on the length of a well-chosen multi-curve has been pushed further by Canary who enhances the arguments of the proof of the Double Limit Theorem to get:

\begin{theorem}[{\cite{canary:schottky}}]	\label{canarschottky}
Let $H$ be a handlebody and consider a sequence $\{\rho_i\}\subset \mathrm{int}(AH(\pi_1(H)))$ with Ahlfors--Bers coordinates converging to a Masur domain lamination. If $H=S\times I$ and $\ell_{\rho_i}(\partial S)\leq K$ for all $i$ and some $K$ independent of $i$, then $\{\rho_i\}$ has a convergent subsequence in $AH(\pi_1(H))$.
\end{theorem}

This statement raises another, although less decisive, issue: deciding what condition will replace the assumption that the laminations are binding. In \cite{masur:domain}, Masur introduced an open subset of $\PML(\partial H)$ for a handlebody $H$ which is now known as the {\em Masur domain}. Save for some exceptional cases, it consists in projective measured laminations which intersect every projective limits of meridians.

It was conjectured by Thurston (see \cite{canary:schottky}) that this domain was the appropriate setting to extend the Double Limit Theorem to handlebodies. Later, this definition was extended to compression bodies by Otal in \cite{otal:these} (see also \cite{kleineidam:souto}).\\

When trying to extend Otal's proof to manifolds with compressible boundary, one also encounters important difficulties. As already explained the issue of extending Lemma \ref{ahlfors} and the assumption that laminations are binding are shared by both proofs. When $M$ is not an $I$-bundle, we still use Culler--Morgan--Shalen's compactification to get a small minimal action of $\pi_1(M)$ on a real tree. If $S$ is a component of $\partial M$, we get an action of $\pi_1(S)$ on the same real tree through the map $i_*:\pi_1(S)\to\pi_1(M)$ induced by the inclusion. But when $S$ is compressible, this action is not small and hence may not be dual to a measured geodesic lamination. Showing that a measured lamination on the boundary is realised in the tree in order to use the Continuity Theorem becomes problematic. Again, one way to get around this obstacle is to assume a control on the length of some multi-curve whose complement is incompressible. With this idea, one can obtain statements that are close to Theorem \ref{canarschottky} (with the limitations described in Remark \ref{almost max}), see \cite{otal:schottky}. Deducing from the work of Culler--Vogtmann \cite{culler:vogtmann} that any action of a rank-$2$ free group is dual to a measured lamination on a compact surface, Otal shows the following: 

\begin{theorem}[{\cite[Theorem 1.5]{otal:schottky}}]
Let $H$ be a genus-$2$ handlebody and $\{\rho_i\}$ a sequence in $\mathrm{int}(AH(\pi_1(H)))$ with Ahlfors--Bers coordinates converging to a Masur domain lamination whose complementary regions are simply connected. Then $\{\rho_i\}$ has a converging subsequence.
\end{theorem}

Before discussing further developments, we should mention the work of Ohshika on free products $\Gamma=\pi_1(S_1)*\pi_1(S_2)$ of two surface groups. In \cite{ohshika:free:product}, he uses the Culler--Morgan--Shalen compactification and a careful study of actions $\Gamma$ on real trees when both surface groups have fixed points to prove a convergence result for representations in $AH(\Gamma)$ whose exterior boundary tend to a Masur domain lamination.\\

The most important breakthrough regarding the convergence of sequences in $AH(\pi_1(M))$ when $M$ has compressible boundary was achieved by Kleineidam--Souto in \cite{kleineidam:souto}. By pursuing the study of limits of meridians initiated in \cite{otal:these} and cleverly adapting some arguments from \cite{skora:splitting}, they manage to prove the following:

\begin{theorem}[{\cite[Corollary 3]{kleineidam:souto}}]
Let $H$ be a handlebody and $\pi_1(H)\times T\to T$ be a non-trivial small minimal action on a real tree $T$. Then at least one minimal component of every measured lamination in the Masur domain is realised in $T$.
\end{theorem}

This allows them to use the Continuity Theorem to show that, for a handlebody $M$, a sequence in $AH(\pi_1(M))$ is precompact if we assume a bound on the length of a sequence of measured laminations converging to a Masur domain lamination. These results are extended to compression bodies in the same paper (\cite{kleineidam:souto} and then to compact atoroidal $3$-manifolds in \cite{lecuire:plissage} (see also \cite[Proposition 6.1 and Theorem 6.6]{lecuire:masur}). This leads to some nice generalizations of the Double Limit Theorem such as
\cite[Theorem 8.1]{namazi:souto:density} and \cite[Theorem 3.8]{ohshika:density} although the need to converge to an almost minimal lamination (see Remark \ref{almost max}) adds some technical hypothesis to the statements.

The final page in this story was written by Kim--Lecuire--Ohshika, \cite{kim:lecuire:ohshika}, who lifted this last limitation with an area argument in a simplicial annulus as in the proof of Efficiency of Pleated Surface and the analysis of limits of boundaries of essential disks and annuli initiated in \cite{otal:these} and pursued in \cite{kleineidam:souto} and \cite{lecuire:plissage}.

\begin{theorem}[{\cite{kim:lecuire:ohshika}}]	\label{klo}
Let $M$ be a compact orientable irreducible atoroidal $3$-manifold, let $\mathcal{C}$ be a connected component of $\mathrm{int}(AH(\pi_1(M))$ containing a representation uniformizing $M$. Let $\{\rho_i\}\subset \mathcal{C}$ be a sequence of convex cocompact representations with Ahlfors--Bers coordinates $\sigma_i\in\mathcal{T}(\partial M)$. If $\sigma_i$ converges to a doubly incompressible measured lamination, then $\{\rho_i\}$ has a convergent subsequence.
\end{theorem}

This statement uses a slight generalization of Masur domain introduced in \cite{lecuire:plissage}: a measured geodesic lamamination $\lambda\in\ML(\partial M)$ is {\em doubly incompressible} if there exists $\eta>0$ such that $i(\lambda,\partial E)>\eta$ for any essential disc or annulus $E\subset M$.

Notice that if $\lambda$ is not doubly incompressible, using Dehn twists along annuli, a  diverging sequence $\{\rho_i\}\subset AH(\pi_1(M))$ can be constructed so that $\sigma_i$ tend to $\lambda$ (compare with \cite{kim:diverge}).


\section{Necessary conditions}  \label{sec:necessity}

The theorems mentioned in the previous sections provide necessary conditions for a sequence to have a converging subsequence. As we already mentioned, Theorems \ref{limit double} and \ref{klo} are optimal in the sense that if a measured lamination $\lambda$ does not satisfy their assumptions, then there is a diverging sequence whose Ahlfors--Bers coordinates tend to $\lambda$. On the other hand there are a lot of converging sequences which do not satisfy the condition of Theorem \ref{klo} (Theorem \ref{limit double} is simply the special case where $M=S\times I$), i.e. their Ahlfors--Bers coordinates $\sigma_i$ tend to a measured lamination that is not doubly incompressible. This condition is far from being necessary. 

In the quasi-Fuchsian case, some necessary conditions have been established by Ohshika with \cite[theorem 3.1]{ohshika:divergent} and \cite[Theorems 3,5 and 12]{ohshika:divergence}. Let $\rho_i:\pi_1(S)\to PSL_2(\C)$ be quasi-Fuchsian representations with Ahlfors--Bers coordinates $(\sigma_i^\pm)$ such that $\{\sigma_i^\pm\}$, converge to a (projective) measured lamination $\mu^\pm$. A very rough description of Ohshika's statements could be that if $\mu^+$ and $\mu^-$ share something (a leaf or a boundary component of a supporting surface), then $\{\rho_i\}$ diverges.

There is a large gap between these necessary conditions and the sufficient conditions of Theorem \ref{klo}. One origin of this gap is the coarseness of Thurston's compactification: it only records the part of the representation that degenerates the fastest. Let us illustrate this idea with an example. Let $S$ be a closed surface and $c, d\subset S$ be two disjoint simple closed curves that are not isotopic and denote by $\psi_c,\psi_d$ the right Dehn twist along $c$, resp. $d$. Fix $X\in \mathcal{T}(S)$ and consider for every $i>0$ the quasi-Fuchsian group $\rho_i:\pi_1(S)\to PSL_2(\C)$ with Ahlfors--Bers coordinates $(\psi_c^{i^i}X,\psi_d^{i^i}X)$. It is easy to deduce from Lemma \ref{ahlfors} that $\{\rho_i\}$ has a converging subsequence. Let $\Psi_c:S\times I\to S\times I$ be the Dehn twist along the annulus $c\times I$. One can prove that $\theta_i=\rho_i\circ(\Psi_{c}^i)_*$ has no converging subsequence. On the other hand, the Ahlfors--Bers coordinates of $\theta_i$, $(\psi_c^i\circ\psi_c^{i^i}X,\psi_c^i\circ\psi_d^{i^i}X)$, have the same limit $(c,d)$ as the coordinates of $\rho_i$.

A way to have a finer compactification of Teichm\"uller space is given by the Culler--Morgan--Shalen compactification by actions on $\Lambda$-trees (see \cite{morgan:shalen:1}) and  their dual $\Lambda$-measured laminations (see \cite{morgan:otal}). A simpler alternative which ought to give similar results for this specific problem is to look at Hausdorff limits of short pants decompositions or short collections of binding curves. Necessary conditions have also been given by Ohshika with this idea, see \cite[Theorem 4]{ohshika:divergence}. But there is still a gap which, as illustrated by \cite[Example 1.4]{bbcl1}, cannot be filled within this framework.\\

Brock, Bromberg, Canary and Lecuire managed to close these gaps in \cite{bbcl1} with a different approach based on Masur--Minsky's work on the curve complex (especially \cite{masur:minsky:2}).

\begin{theorem}[{\cite{bbcl1}}]    	\label{bbcl}
Let $S$ be a compact, orientable surface and let $\{\rho_i\}$ be a
sequence in $AH(\pi_1(S))$ with Ahlfors--Bers coordinates $\{\sigma_i^\pm\}$. Then
$\{\rho_i\}$ has a convergent subsequence if and only if 
there exists a subsequence $\{\rho_j\}$ of $\{\rho_i\}$ such
that $\{\sigma_j^\pm\}$ bounds projections.
\end{theorem}

The statement is short because the authors have craftily hidden the technicalities in the definition of ``bounds projections". To give the precise definition would require too many preliminaries but we will try to convey the spirit. The definition of \lq\lq bounding projections" consists in two conditions that we will refer as (a) and (b), following \cite{bbcl1}. Condition (a) essentially prevents the case where both $\{\sigma_i^+\}$ and $\{\sigma_i^-\}$ tend to filling projective measured laminations with the same support (see also \cite{ohshika:divergent}). Condition (b) sees the introduction of {\em combinatorial parabolics}. Those are simple curves on $S$ for which the behavior of $\sigma_i^+$ (for upward pointing combinatorial parabolics) or $\sigma_i^-$ (for downward pointing ones) indicates that they should be parabolics in the limit (if there was one). Condition (b) essentially says that a simple closed curve on $S$ cannot be simultaneously an upward pointing and a downward pointing combinatorial parabolic. The possibility of wrapping of the algebraic limit, as described in \cite{anderson:canary:rearrange}, compels us to add some exceptions to this last condition (condition (b)(ii), see also \cite[Theorem 6]{ohshika:divergence}).\\

The proof of Theorem \ref{bbcl} as well as the proofs of the main Theorems in \cite{ohshika:divergence} make extensive use of the works of Masur--Minsky and Minsky (\cite{masur:minsky:2} and \cite{elc1}). Masur and Minsky associate a family of simple closed curves $\mathcal{H}^0_\nu$, to a pair of end invariants $\{\sigma^\pm\}$. They add some structure to $\mathcal{H}^0_\nu$ to form what they call a {\em hierarchy} $\mathcal{H}_\nu$. Minsky builds from this hierarchy a model $M_\nu$, i.e. a piecewise Riemannian manifold homeomorphic to $S\times I$ whose metric depends only on the hierarchy $\mathcal{H}_\nu$. Then with the collaboration of Brock and Canary, Minsky (\cite{elc1} and \cite{elc2}) shows that for any hyperbolic manifold $N_\rho$ with end invariants $\{\sigma^\pm\}$, there is a bilipschitz map $M_\nu\to N_\rho$. When $\rho$ is convex cocompact, the end invariants are the Ahlfors--Bers coordinates. In general they are a mixture of conformal structure at infinity and ending laminations which describe the asymptotic behavior of the geometry of the ends of $\Hp^3/\rho(\pi_1(S))$. This work on the models led to the proof of Thurston's Ending Lamination Conjecture which asserts that a representation $\rho\in AH(\pi_1(S))$ is uniquely defined by its end invariants. One important result of \cite{elc1} that makes the construction of the model work is the existence of a bound on the length in $N_\rho$ of all the curves of $\mathcal{H}^0_\nu$. Furthermore, this bound depends only on $S$.

Let us go back to Theorem \ref{bbcl} and consider a sequence $\{\rho_i\}$ in $AH(\pi_1(S))$ with Ahlfors--Bers coordinates $\{\sigma_i^\pm\}$. When $\{\sigma_i^\pm\}$ bounds projections, Brock--Bromberg--Canary--Lecuire deduce from the structure of $\mathcal{H}^0_{\sigma_i}$ that it contains a family of binding curves independent of $i$. The convergence follows immediately. On the other hand, Brock--Bromberg--Canary--Minsky (\cite{bbcm:pullout}) and Ohshika (\cite{ohshika:divergence}) use Minsky's model to study the link between the behavior of the end invariants of a sequence and the end invariants of a limit. This leads to the divergence results in \cite{ohshika:divergence} and the necessity part in Theorem \ref{bbcl}.\\

We conclude this section by noticing that all the results we have mentioned so far give conditions for a convergence up to extracting a subsequence. To have convergence of the actual sequence would mean to completely predict the end invariants of the limit. The fact that the geometric limit often differs from the algebraic limit makes such a prediction extremely difficult.

\section{Some applications} \label{sec:applications}

The original motivation for the Double Limit Theorem and the compactness of $AH(\acylindrical)$ was the Hyperbolization Theorem for Haken manifolds:

\begin{theorem}[Hyperbolization Theorem]
Let $M$ be a compact irreducible atoroidal Haken $3$-manifold, then the interior of $M$ has a complete hyperbolic structure.
\end{theorem}

Even though Thurston never published a complete proof for reasons that he explained in \cite{thurston:proof}, he shared his arguments on multiple occasions and wrote some of them in \cite{thurston:hypI}, \cite{thurston:hypII} and \cite{thurston:hypIII}. The proof decomposes into two distinct cases, each one using a different convergence result.

In the case of manifolds that fiber over the circle, the Double Limit Theorem is used to construct an invariant metric on the cyclic cover. Thurston's arguments have been summarized in \cite{sullivan:hyperbolization:fibred} and Otal wrote a complete proof in \cite{otal:fibre} (with different arguments to prove the Double Limit Theorem, as explained in \textsection \ref{otal}).

In the other case, the compactness of $AH(\acylindrical)$ is used to establish the Bounded Image Theorem (see \cite[Theorem 41]{kent:skinning}) which allows hyperbolic pieces to be glued together to form a larger hyperbolic manifold. Morgan summarized Thurston's arguments in \cite{morgan:hyperbolization} and Kapovich, \cite{kapovich:book}, and Otal, \cite{otal:haken}, wrote complete proofs.\\

The Density Theorem is another example of a proof in which convergence results play an important role. 

\begin{theorem}[Bers--Thurston's Density]  \label{bers-thurston}
Every finitely generated Kleinian group is an algebraic limit of geometrically finite groups.
\end{theorem}

This statement resolves a generalization due to Sullivan and Thurston of a conjecture of Bers. Combined with works of Marden and Sullivan, Theorem \ref{bers-thurston} shows that the deformation space $AH(\pi_1(M))$ does not have any isolated point. Its proof has been written out by Namazi--Souto, \cite{namazi:souto:density}, and Ohshika \cite{ohshika:density}, it uses the Tameness Theorem (\cite{agol:tameness} and \cite{calegari:gabai}), the Ending Lamination Theorem (\cite{elc1}, \cite{elc2}), a convergence Theorem (for example Theorem \ref{klo}, but a weaker statement is sufficient) and an additional argument to show that non-realizable laminations are ending laminations (see \cite[Theorem 1.4]{namazi:souto:density} or \cite[Proposition 6.5]{ohshika:density}). The fact that this proof uses the resolutions of two difficult conjectures is a good illustration of how unfathomable the topology of the deformation space $AH(\pi_1(M))$ can be. Notice that an alternate approach has been developed by Brock--Bromberg \cite{brock:bromberg:density} when $\partial M$ is incompressible.\\

Combining the Density Theorem with the Ahlfors--Bers coordinates, we get that $AH(\pi_1(M)$ is the closure of an union of topological balls (assuming that $\partial M$ is incompressible to simplify the statements). Despite this apparent simplicity, various exotic phenomenons have been observed. First Anderson--Canary, \cite{anderson:canary:rearrange} proved that two of those balls may have intersecting closures. Then McMullen, \cite{mcmullen:bumping} and Bromberg--Holt, \cite{bromberg:holt:bumping} showed that those balls may self-bump. Lastly, Bromberg, \cite{bromberg:conloc}, and Magyd \cite{magid:conloc}, concluded that $AH(\pi_1(M))$ may not be locally connected. On the other hand by studying the ways different sequences converge to a point, we can find points where none of these happen:

\begin{theorem}[{\cite{bbcm:localtop} and \cite{bbclm:localtop}}]
Let $M$ be a compact atoroidal $3$-manifold with incompressible boundary. If $\rho$ is a quasiconformally rigid point in $\partial AH(\pi_1(M))$ then $\rho$ is uniquely approachable. In particular, $AH(\pi_1(M))$ is locally connected at $\rho$ and there is no self--bumping at $\rho$.
\end{theorem}

A representation $\rho$ is {\em quasiconformally rigid} if $\Omega_\rho/\rho(\pi_1(M))$ is a union of three holed spheres.\\ 

The proofs of many more results could illustrate the usefulness of the convergence results presented in this chapter. To drive this point home, let us also mention the work of Bonahon--Otal \cite{bonahon:otal:plissage} and Lecuire \cite{lecuire:plissage} on bending measured laminations and the work of Namazi \cite{namazi:thesis}, Namazi--Souto \cite{namazi:souto:heegaard:splitting} and Brock--Minsky--Namazi--Souto \cite{bmns:bounded:combinatorics}
 on models for compact and non-compact hyperbolic $3$-manifolds.
 
 We would like to conclude this chapter by mentioning an article of Biringer--Souto, \cite{biringer:souto:unfaithful}, where the authors study sequence of unfaithful representations.

\bibliographystyle{abbrv}
\bibliography{refs}

\end{document}